\documentclass[12pt]{amsart}

\usepackage{verbatim}
\usepackage{hyperref}

\newtheorem{theorem}{Theorem} [section]
\newtheorem{prop}[theorem]{Proposition}
\newtheorem{lemma}[theorem]{Lemma}

\theoremstyle{definition}

\theoremstyle{remark}
\newtheorem*{remark}{Remark}

\numberwithin{equation}{section}
\numberwithin{figure}{section}
\numberwithin{example}{section}

\newcommand\C{{\mathbb C}}

\renewcommand\P{{\mathbb P}}

\newcommand\R{{\mathbb R}}

\newcommand\Q{{\mathbb Q}}
\newcommand\N{{\mathbb N}}
\newcommand\D{{\mathbb D}}

\newcommand\eps{\varepsilon}

\renewcommand\phi{\varphi}

\newcommand{\bC}{\mathbb{C}}

\newcommand{\bR}{\mathbb{R}}
\newcommand{\bN}{\mathbb{N}}
\newcommand{\bD}{\mathbb{D}}

\newcommand\Qbar{\overline{\mathbb{Q}}}

\begin{document} 

\title
{Discontinuity of a degenerating escape rate}

\author{Laura DeMarco}
\address{Department of Mathematics, Northwestern University, USA}
\email{demarco@math.northwestern.edu}

\author{Y\^usuke Okuyama}
\address{Division of Mathematics,
Kyoto Institute of Technology, Japan}
\email{okuyama@kit.ac.jp}


\date{\today}

\begin{abstract}
We look at degenerating meromorphic families of rational maps on $\mathbb{P}^1$ -- holomorphically parameterized by a punctured disk -- and we provide examples where the bifurcation current fails to have a bounded potential in a neighborhood of the puncture.  This is in contrast to the recent result of Favre-Gauthier that we always have continuity across the puncture for families of polynomials; and it provides a counterexample to a conjecture posed by Favre in 2016.  We explain why our construction fails for polynomial families and for families of rational maps defined over finite extensions of the rationals $\mathbb{Q}$. 
\end{abstract}


\thanks{This research was partially supported by JSPS Grant-in-Aid for Scientific Research (C), 15K04924, and the National Science Foundation DMS-1600718.}

\maketitle 

\section{Introduction}\label{sec:intro}

Let $f_t$ be a holomorphic family of rational maps on $\P^1$ of degree $d>1$, parameterized by the punctured unit (open) disk $\D^* = \{t\in\C:0 < |t| < 1\}$, and assume that the coefficients of $f_t$ extend to meromorphic functions on the unit disk $\D= \{t\in\C: |t| < 1\}$.  Let $a: \D \to \P^1$ be a holomorphic map.  In this article, we examine the potential function $g_{f,a}$ on $\D^*$ (having the order $o(\log|t|)$ as $t\to 0$) for the bifurcation measure associated to the pair $(f,a)$.  Our main result is that this potential function does not necessarily extend continuously across the puncture at $t=0$.  

The question of continuous extendability of $g_{f,a}$ across $t=0$ arose naturally in the study of degenerating families of rational maps, and specifically in the context of equidistribution questions and height functions associated to the family $f_t$; see, e.g., \cite{BD:polyPCF,Favre:degenerations}.  Continuity of the potential at $t=0$ was required to apply certain equidistribution theorems on arithmetic varieties (as in the proofs of the main results of \cite{BD:polyPCF, DM:variation,Favre:Gauthier:cubics,Ghioca:Ye:cubics}, and others).  Moreover, when $a(t)$ parameterizes a critical point of $f_t$, the bifurcation measure and its potential are related to the structural stability of the family $f_t$ \cite{D:current,Dujardin:Favre:critical}.  It is well known that continuity holds when $f_t$ has a uniform limit on the whole $\P^1$ as $t\to 0$, for any choice of $a$.
It is also true when $f_t$ is any family of polynomials with coefficients meromorphic in $t$, again for any choice of $a$ \cite{Favre:Gauthier:continuity}

To formulate the problem and our construction more precisely, 
we will work with $f_t$ in homogeneous coordinates: assume that we are given a family of homogeneous polynomial maps $\tilde{f}_t: \C^2\to\C^2$ of degree $d$, where the coefficients are holomorphic functions on the entire disk $\D$, such that for every $t\in\D^*$, 
$\tilde{f}_t^{-1}(0,0) =  \{(0,0)\}$
and $\tilde{f}_t$ projects to $f_t$ on $\P^1$.  There exists a continuous plurisubharmonic escape rate
$$G_{\tilde{f}}: \D^*\times(\C^2 \setminus\{(0,0)\}) \to  \R
$$ 
such that for each fixed $t \in \D^*$, the current $dd^c G_{\tilde{f}}(t, \cdot)$ on $\bC^2\setminus\{(0,0)\}$ projects to the measure of maximal entropy of $f_t$ on $\P^1$ \cite{Hubbard:Papadopol, Fornaess:Sibony}.  Given a holomorphic lift of $a$ to $\tilde{a}: \D\to \C^2\setminus\{(0,0)\}$, we may write
\begin{equation} \label{G near 0}
	G_{\tilde{f}}(t,\tilde{a}(t)) = \eta \log|t| + g_{f,a}(t),
\end{equation}
where $\eta\in \R$ represents a ``local height" for the pair $(f,a)$, and 
the function $g_{f,a}$ on $\D^*$ satisfies
$$ g_{f,a}(t) = o(\log|t|) $$
as $t\to 0$ \cite{D:stableheight}; see \S\ref{subharmonic extension}.
The value of $\eta$ and the subharmonic function $g_{f,a}$ depend on the choices of $\tilde{f}$ and $\tilde{a}$, but $g_{f,a}$ is uniquely determined up to the addition of a harmonic function on $\D^*$ which is bounded near $t=0$.  The Laplacian $\mu_{f,a} = \frac{1}{2\pi}\Delta g_{f,a}$ on $\D^*$ is the {\em bifurcation measure associated to the pair} $(f,a)$ \cite[\S 3]{Dujardin:Favre:critical}.

It turns out that the function $g_{f,a}$ is always bounded from above near $t=0$ (Lemma \ref{subharmonic}).  In this article, we construct examples of pairs $(f,a)$ that satisfy 
	$$\limsup_{t\to 0} g_{f,a}(t) = -\infty$$
to show that it need not be bounded from below, so in particular does not extend continuously across $t=0$. In our examples, the maps $f_t$ will converge to a rational map $\phi$ on $\P^1$ of degree $< d$ as $t\to 0$ locally uniformly on $\P^1\setminus H$, where $H$ is a non-empty finite set.  The idea of the construction is to choose $\phi$ and $a$ so that some sequence of iterates $\phi^{n_j}(a(0))$ accumulates fast on $H$ as $n_j\to \infty$. 

Furthermore, choosing $a(t)$ to parameterize a critical point of  the family $f_t$, we obtain a counterexample to the continuity statement in \cite[Conjecture 1]{Favre:degenerations}, in proving:

\begin{theorem} \label{discontinuity theorem}
For every integer $d>1$, there exists a holomorphic family $f_t$ of rational maps on $\P^1$ of degree $d$,
parameterized by $t \in \D^*$, whose coefficients extend to meromorphic functions on $\D$ but for which the bifurcation current associated to the family $f_t$ fails to have a bounded potential in any punctured neighborhood of $t = 0$.  
\end{theorem}

\begin{remark}  
It will be clear from the proof that the family $f_t$ can be chosen to be algebraic, in the sense that it extends to define a holomorphic family parameterized by $t$ in a quasiprojective curve $X$, with coefficients that are meromorphic on a compactification of $X$. 
\end{remark}

The bifurcation current associated to the family $f_t$ is equal to the Laplacian of the continuous and subharmonic function $t\mapsto L(f_t)$ on $\D^*$, where for each $t$, $L(f_t)$ is the Lyapunov exponent of $f_t$ with respect to its unique measure of maximal entropy. For more details on $L(f_t)$ and its relationship to $g_{f,a}$, see Section \ref{Lyapunov}.

The construction of examples of pairs $(f,a)$ for which $g_{f,a}$ fails to extend continuously across $t=0$ is laid out in Section \ref{recipe}.  Our use of the Baire category theorem in the construction is similar to that of \cite[Example 4]{Favre:1998}, \cite{Buff:courants}, or \cite{Diller:Guedj:regularity} in the context of higher-dimensional (bi)rational maps.
In Section \ref{Lyapunov}, we give the proof of Theorem \ref{discontinuity theorem}.  In Section \ref{limitations}, we comment on why the strategy for producing these examples fails for families of polynomials and for rational maps on $\P^1$ defined over $\Qbar$.  We expect that a continuous extension of $g_{f,a}$ to $\D$ always exists when the pair $(f,a)$ is algebraic and defined over $\Qbar$, as is known for algebraic families of elliptic curves \cite[Theorem II.0.1]{Silverman:VCHII} and therefore also for Latt\`es maps on $\P^1$ \cite[Proposition 3.4]{DM:variation}; see also \cite[Theorem A]{Jonsson:Reschke} in the context of (bi)rational maps in dimension 2.  The bifurcation current associated to a family $f$ was introduced in \cite{D:current}; its properties at infinity in the moduli space of quadratic rational maps (related to our Theorem \ref{discontinuity theorem}) were studied in \cite{Berteloot:Gauthier}.

We would like to thank Charles Favre, Thomas Gauthier, and the anonymous referee for helpful comments and suggestions.


\bigskip
\section{A recipe for discontinuity}
\label{recipe}

In this section, we construct the examples for which $g_{f,a}$ fails to extend continuously to the disk $\D$. 

\subsection{The potential is bounded from above}  \label{subharmonic extension}
Suppose we are given a family of homogeneous polynomial maps $\tilde{f}_t: \C^2\to\C^2$ of degree $d>1$, where the coefficients are holomorphic functions on the entire disk $\D$, and such that for every $t\in\D^*$, we have $\tilde{f}_t^{-1}(0,0) =  \{(0,0)\}$ so $\tilde{f}_t$ projects to a rational map $f_t$ 
on $\P^1$ of degree $d$.  We let $\tilde{a}: \D\to \C^2\setminus\{(0,0)\}$ be any holomorphic map, and let $a: \D\to \P^1$ be its projection.  For each $n\in\N$, there is a unique non-negative integer $o_n$ so that 
	$$F_n(t) := t^{-o_n} \tilde{f}_t^n(\tilde{a}(t))$$ 
is a holomorphic map from $\D$ to $\C^2\setminus\{(0,0)\}$.  
Choose any norm $\|\cdot\|$ on $\C^2$. The function $g_{f,a}$ on $\D^*$ defined by \eqref{G near 0} is the locally uniform limit on $\D^*$ of 
the sequence of continuous and subharmonic functions 
	$$g_n(t) :=  \frac{1}{d^n} \log\|F_n(t)\|\quad  \text{on }\D,$$
as $n\to\infty$, and the value $\eta$ of \eqref{G near 0} is given by 
	$$\eta = \lim_{n\to\infty}\; \frac{o_n}{d^n},$$ 
as explained in \cite[\S3]{D:stableheight}. 
Note, in particular, that the function $g_{f,a}$ is continuous and subharmonic on $\D^*$.

The following observation is not required in this section, but it will be useful in Section \ref{Lyapunov}. 

\begin{lemma} \label{subharmonic}
The function $g_{f,a}$ is bounded from above on $\{0<|t|\le r\}$ for every $r\in(0,1)$.
\end{lemma}

\begin{proof}
Fix any $r\in(0,1)$.  We know that $\lim_{n\to\infty}g_n=g_{f,a}$ uniformly on the circle $C_r = \{|t| = r\}$.  Let $M_r = \max_{C_r} g_{f,a}$.  Then, for all $n$ large enough, we have $g_n \leq M_r + 1$ on $C_r$.  As the functions $g_n$ are subharmonic 
on $\D$, we also have $g_n \leq M_r + 1$ on the disk $\{|t| \leq r\}$ for all $n$ large enough.  It follows that $g_{f,a} \leq M_r + 1$ on the punctured disk $\{0 < |t| \leq r\}$.
\end{proof}

\begin{remark}  
As an immediate consequence of Lemma \ref{subharmonic}, 
the function $g_{f,a}$ could extend to a subharmonic function on the disk $\D$.  

In this article,
we would dispense with this possible extension of $g_{f,a}$, so that
the domain of definition of $g_{f,a}$ is kept being the punctured disk $\bD^*$.
\end{remark}

\subsection{The ingredients for discontinuity} \label{ingredients}
Let $\phi\in\bC(z)$ be a rational map on $\P^1$ of degree $e \geq 1$, and suppose that there is a point $a_0\in\C$ such that $\#\{\phi^n(a_0):n\in\bN\}=\infty$ and that $\omega_\phi(a_0)\cap\{\phi^n(a_0):n\in\bN\}\neq\emptyset$, where 
\begin{gather*}
 \omega_\phi(a_0):=\bigcap_{N\in\bN}\overline{\{\phi^n(a_0):n>N\}}
\end{gather*}
is the $\omega$-limit set of $a_0$ under $\phi$.  Then there exists $N_0\in\bN$ so that $\{\phi^n(a_0):n\geq N\}$ is dense in $\omega_\phi(a_0)$ for all $N\geq N_0$.  

Let $\{r_n\}$ be any sequence in $\bR_{>0}$ decreasing to $0$ as $n\to\infty$, which will be chosen appropriately later.   It follows that the set
	$$U_N(a_0, \{r_n\}) := \left( \bigcup_{n\geq N}\{z\in \omega_\phi(a_0): [z,\phi^n(a)]<r_n\}\right)
		 \setminus \{a_0, \phi(a_0), \ldots, \phi^{N-1}(a_0)\}$$ 
is open and dense in $\omega_\phi(a_0)$ for all $N\geq N_0$.  Here $[\cdot,\cdot]$ denotes the chordal distance on $\P^1$.  Therefore, by the Baire category theorem, 
	$$ B_\phi(a_0, \{r_n\}):=\bigcap_{N\geq N_0} U_N(a_0, \{r_n\})$$
is dense in $\omega_\phi(a_0)$. 

Fix any $h\in B_\phi(a_0, \{r_n\})\cap\bC$.  Then $\phi^n(a_0) \not= h$ for all $n\in\N\cup\{0\}$, and there is a sequence $n_j \to \infty$ such that 
\begin{gather}	
 	0<[\phi^{n_j}(a_0),h]<r_{n_j}\label{eq:derived}
\end{gather}
for all $j\in \N$.  

We consider the family
\begin{gather} \label{f definition}
 f_t(z):=\phi(z)\cdot\frac{z-h - \eps t}{z- h +\eps t}
\end{gather}
parameterized by $t \in \D^*$, where $\eps>0$ is chosen so that $\phi$ has neither zeros nor poles in the set $\{z: 0 < |z-h| < \eps\}$.  Thus, $f_t$ defines a holomorphic family of rational maps of degree $d:= e+1 >1$.  As $t\to0$, the maps $f_t$ converge locally uniformly to $\phi$ on $\P^1\setminus\{h\}$.

\subsection{An unbounded escape rate}
Set now
	$$r_n=\exp(-n\,d^{n+1})$$ 
for each $n\in\bN$.  
Working on $\C^2$, we define 
	$$\tilde{f}_t(z,w) := (\, P(z,w)(z-(h+\eps t)w), \, Q(z,w)(z-(h-\eps t)w) \,)$$
for all $t\in \D$, where $P$ and $Q$ are homogeneous polynomials of degree $e = \deg \phi$ such that $\phi(z) = P(z,1)/Q(z,1)$.  Let $\tilde{a}: \D\to \C^2\setminus\{(0,0)\}$ be any holomorphic map such that $\tilde{a}(0) = (a_0,1)$
and let $a: \D\to \P^1$ be its projection to $\P^1$.

Choose any norm $\|\cdot\|$ on $\C^2$.   As $\phi^n(a_0) \not= h$ for all $n\geq 0$, we see that $\tilde{f}_{0}^{n}(\tilde{a}(0)) \not= (0,0)$ for all $n\geq 0$.  Therefore, as described in \S\ref{subharmonic extension}, we have $\eta = 0$ and the function $g_{f,a}$ is given by the formula 
\begin{equation} \label{good g}
	g_{f,a}(t) = \lim_{n\to\infty} \frac{1}{d^n} \log \| \tilde{f}_t^n(\tilde{a}(t)) \|
\end{equation}
for $t\in \D^*$ \cite[Proposition 3.1]{D:stableheight}.

Set
 $$\Phi  := (P,Q) \quad\mbox{and}\quad H(z,w) := z-hw$$
so that $\tilde{f}_0 = (HP, HQ)$.  For all $n \geq 0$,  
as $\deg \Phi = e >0$, the iteration formula of \cite[Lemma 2.2]{D:measures} states that
$$\tilde{f}_0^n = \left( P_n \cdot \prod_{k=0}^{n-1}( (\Phi^k)^* H)^{d^{n-k-1}} , \; Q_n \cdot \prod_{k=0}^{n-1} ((\Phi^k)^* H)^{d^{n-k-1}} \right),$$
where we set $\Phi^n = (P_n, Q_n)$, so that 
$$ \frac{\log\|\tilde{f}_0^n\|}{d^n} =\sum_{k=0}^{n-1}\frac{\log|(\Phi^k)^*H|}{d^{k+1}}
+\frac{\log\|\Phi^n\|}{d^n}\quad\text{on }\bC^2\setminus\{(0,0)\}, $$
and consequently, 
\begin{equation} \label{distance term}
 \log\frac{\|\tilde{f}_0\circ \tilde{f}_0^n\|}{\|\tilde{f}_0^n\|^d}
=\log\frac{|(\Phi^n)^*H|}{\|\Phi^n\|}+\log\frac{\|\Phi\circ \Phi^n\|}{\|\Phi^n\|^e}
\quad \text{on }\bC^2\setminus \tilde{f}_0^{-n}(0,0).
\end{equation}

Note that $\log \|\Phi\|$ is bounded on the unit sphere in $\C^2$, so the last term on the right-hand side of \eqref{distance term} is 
bounded on $\bC^2\setminus \{(0,0)\}$
uniformly in $n \ge 0$.  The first term on the right-hand side of \eqref{distance term} is the $\log$ of $[\phi^n(\cdot),h]$,
up to scaling of the metric $[\cdot,\cdot]$; therefore, combined with \eqref{eq:derived}, we see that there is a constant $C$ so that
 $$ \log \frac{\|\tilde{f}_0(\tilde{f}_0^{n_j}(\tilde{a}(0)))\|}{\|\tilde{f}_0^{n_j}(\tilde{a}(0))\|^d} < C + \log (r_{n_j}) = C - n_j d^{n_j+1}$$
for all $j$.  For all $j$, by continuity of $\tilde{f}^{n_j}_t(\tilde{a}(t))$ as a map from $\D$ to $\C^2\setminus\{(0,0)\}$, there is a radius $\delta_j\in(0,1/2)$ such that 
\begin{equation} \label{estimate near 0}
	\sup_{|t| \leq \delta_j} \; \log \frac{\|\tilde{f}_t(\tilde{f}_t^{n_j}(\tilde{a}(t)))\|}{\|\tilde{f}_t^{n_j}(\tilde{a}(t))\|^d} \leq   C - n_j d^{n_j+1}.
\end{equation}

On the other hand, we also have from \eqref{good g} that
\begin{eqnarray} \label{for upper bound}
 g_{f,a}(t) 
 	&=& \log\|\tilde{a}(t)\|+\sum_{k=0}^{\infty}\frac{1}{d^{k+1}}\log\frac{\|\tilde{f}_t(\tilde{f}_t^k(\tilde{a}(t)))\|}{\|\tilde{f}_t^k(\tilde{a}(t))\|^d}   \nonumber  \\
	&=& \log\|\tilde{a}(t)\|+\frac{1}{d^{n_j+1}} \log\frac{\|\tilde{f}_t(\tilde{f}_t^{n_j}(\tilde{a}(t)))\|}{\|\tilde{f}_t^{n_j}(\tilde{a}(t))\|^d}
+\sum_{k\neq n_j}\frac{1}{d^{k+1}}\log\frac{\|\tilde{f}_t(\tilde{f}_t^k(\tilde{a}(t)))\|}{\|\tilde{f}_t^k(\tilde{a}(t))\|^d}
\end{eqnarray}
for each $j$.  

The following is elementary but useful:  
\begin{lemma} \label{upper bound lemma}
Let $F_t = (P_t, Q_t)$ be any family of homogeneous polynomial maps of degree $d \geq 2$, with coefficients that are bounded holomorphic functions of $t$ in $\D$.  Then there is a constant $C$ so that 
	$$\frac{\|F_t(z,w)\|}{\|(z,w)\|^d} \leq C$$
for all $(z,w) \in \C^2\setminus\{(0,0)\}$ and all $t\in \D$.
\end{lemma}

\begin{proof}
As $F_t$ is homogeneous, it suffices to bound its values on the unit sphere in $\C^2$.  The result follows because the coefficients are bounded uniformly on $\D$.
\end{proof}

As a consequence of Lemma \ref{upper bound lemma}, we can bound all the terms in the final sum of \eqref{for upper bound} from above,  uniformly on the disk $\{|t| \leq 1/2\}$, 
and therefore there is a constant $C'$ so that 
\begin{equation} \label{before choosing r}
	\sup_{|t|\leq 1/2} \; g_{f,a}(t) \leq C' + \frac{1}{d^{n_j+1}} \log\frac{\|\tilde{f}_t(\tilde{f}_t^{n_j}(\tilde{a}(t)))\|}{\|\tilde{f}_t^{n_j}(\tilde{a}(t))\|^d}
\end{equation}
for every $j$.  Combined with \eqref{estimate near 0}, we conclude that there is another constant $C$ so that 
	$$\sup_{|t| \leq \delta_j} \; g_{f,a}(t) \leq   C - n_j $$
for every $j$.  Letting $j\to \infty$ shows that 
	$$\limsup_{t\to 0} g_{f,a}(t) = -\infty.$$

\subsection{Examples with degree $d=2$}
Fix any $\theta \in \R\setminus \Q$, and let 
	$$\phi(z) = e^{2\pi i \theta} z.$$
Set $a_0 = 1$; the $\omega$-limit set $\omega_\phi(a_0)$ is the unit circle in $\bC$.  Set $r_n = \exp(-n \, 2^{n+1})$ for each $n\in\N$, and define $B_\phi(1, \{r_n\})$ as above.  Taking any $h \in B_\phi(1, \{r_n\})$ and setting $\eps=1$, we define the family $f_t$ as in \eqref{f definition}.  Then the potential function $g_{f,a}$ fails to be bounded around $t=0$ for any holomorphic map $a: \D\to \P^1$ with $a(0) = 1$.

Note that only the M\"obius transformations $\phi$ 
which are M\"obius (i.e., $\mathrm{PSL(2,\bC)}$-) conjugate to an irrational rotation have recurrent orbits, as needed for the construction described above.

\subsection{Examples in degree $>2$, with a marked critical point} \label{cubic}
Fix an integer $d>2$. For every $\theta\in\bR$, the polynomial
\begin{equation} \label{funny phi}
\phi(z) = e^{2\pi i \theta} \left( z - \frac{e^{2\pi i\theta} - (d-1)}{(d-1)^{(d-1)/(d-2)}}  \right)^{d-1}
\end{equation}
of degree $d-1$
has a fixed point with multiplier $e^{2\pi i \theta}$ and its unique finite critical value at $z=0$.  Now fix $\theta$ to be irrational; the critical point 
	$$a_0 := \frac{e^{2\pi i \theta} - (d-1)}{(d-1)^{(d-1)/(d-2)}}$$ 
of $\phi$ 
satisfies $\#\{\phi^n(a_0):n\in\bN\}=\infty$ and
$a_0 \in \omega_\phi(a_0)$ \cite{Mane:Fatou}.  Let
	$$r_n = \exp(-n \, d^{n+1})$$
for each $n\in\N$ and fix any point $h \in B_\phi(a_0, \{r_n\})$.  
Choose any $\eps\in(0,|a_0 - h|]$, and set
	 $$f_t(z)=\phi(z)\cdot\frac{z-h-\eps \, t}{z-h+\eps \, t},$$
which is a rational map on $\P^1$ of degree $d$ for all $t\in \D^*$.  We let 
	$$a(t) = a_0$$
for all $t \in \D$, which satisfies $f_t'(a(t)) = 0$ for all $t\in \D^*$.  (This is the reason for requiring the unique finite critical value $\phi(a_0)$ of $\phi$ to be 0.)  It follows that $g_{f,a}$ fails to  be bounded around $t=0$.  

\subsection{Example in degree 2, with a marked critical point} \label{quadratic}
We can produce examples of $(f,a)$ also for quadratic rational maps $f_t$ where $a(t)$ parameterizes a critical point of $f_t$, though we do not have as much flexibility as in higher degrees.  For example, we have:

\begin{lemma}
Suppose $f_t(z) = \phi(z) (z-h-t^n)/(z-h+t^n)$ is a family of quadratic rational maps, for some $h \in \C$, $n \in \N$, and a rational map $\phi$ on $\P^1$
of degree 1, and suppose that  $c_1,c_2:\D\to\P^1$ are holomorphic maps parameterizing the two critical points of $f_t$. Then 
	$$\lim_{t\to 0} c_1(t) = \lim_{t\to 0} c_2(t) = h.$$
\end{lemma}

\begin{proof}
As $\phi$ has degree 1, it has no critical points of its own.  On the other hand, 
$\lim_{t\to 0}f_t=\phi$ locally uniformly on $\P^1\setminus\{h\}$, so it must be that $c_1(t), c_2(t) \to h$ as $t\to 0$.
\end{proof}

In particular, if we wish to let $a(t)$ parameterize a critical point of $f_t$, then necessarily we will have $a(0) = h$, which was not allowed by the construction above.  

However, let us fix our decreasing sequence as
	$$r_n = \exp(-(n-1) \, 2^n),$$
for each $n\in\N$, and 
apply the Baire Category Theorem now to the space of rotations 
$z\mapsto e^{2i\pi\theta}z$ 
to find 
	$$\theta_0 \; \in \; \left(\bigcap_{N\in\bN} \bigcup_{n\geq N} \{\theta\in\R : [1, e^{2\pi i n \theta}] < r_n\}\right) \setminus \Q.$$
Then $e^{2\pi i n \theta_0} \not= 1$ for all $n\in\N$, and there is a sequence $n_j \to \infty$ as $j\to\infty$ such that
	$$ 0< [1, e^{2\pi i n_j \theta_0}] < r_{n_j}$$
for all $j$.  

Now set $\phi_0(z) = e^{2\pi i \theta_0} z$, and 
	$$f_t(z)= \phi_0(z) \cdot\frac{z-1-t^2}{z-1+t^2}$$
for $t\in \D^*$.  Note that we have used $t^2$ here rather than $t$ in \eqref{f definition}; this is so that we can holomorphically parameterize the critical points of $f_t$.  Indeed, 
the critical points  of $f_t$ are 
	$$c_{\pm}(t) = 1 - t^2 \pm \sqrt{2t^4 - 2 t^2} = 1 - t^2 \pm i \sqrt{2} \,  t \sqrt{1 - t^2},$$
which extend holomorphically on $\D$ by setting $c_{\pm}(0)=1$. Define the function $a:\bD\to\C$ by either $c_+$ or $c_-$
so that $f_t$ has the critical value
	$$v(t) := f_t(a(t)) = e^{2\pi i \theta_0} + O(t)\quad  \text{as }t\to 0,$$
which also extends holomorphically to $\D$ by setting $v(0):=v_0:= e^{2\pi i \theta_0} = \phi_0(1)$.
We also set 
	$$\tilde{f}_t(z,w) := (e^{2\pi i \theta_0} z (z - (1+t^2) w), (z-(1-t^2)w)w)$$
for all $t\in \D$.

We will work with the pair $(f,v)$.  Then, since $\phi_0^n(v_0) \not=1$ 
for all $n\in\N\cup\{0\}$ and since
we also have 
	$$0 < [1, \phi_0^{n_j-1}(v_0)] < r_{n_j} = \exp(-(n_j-1) \, 2^{(n_j-1)+1})$$
for all $j$, the arguments above go through exactly as before -- applied to the sequence $\{n_j-1\}_j$ -- to show that $g_{f,v}$ fails to be bounded 
around $t=0$.  

Finally, if we set $\tilde{v}(t) = (v(t), 1)$ and $\tilde{a}(t) = (a(t), 1)$, then 
	$$\tilde{f}_t(\tilde{a}(t)) = (a(t)-1+t^2) \, \tilde{v}(t).$$
Note that $a(t)-1+t^2 =  \pm i \sqrt{2} \,  t \sqrt{1 - t^2}$  on $\D$, so that the function
	$$h(t) = \log|a(t)-1+t^2| - \log|t|$$
on $\D^*$ extends to a harmonic function on the disk $\D$. 
We have
	$$G_{\tilde{f}}(t, \tilde{v}(t)) = g_{f,v}(t)$$
on $\D^*$ from the definitions given in \eqref{G near 0} and because the pair satisfies the hypotheses for \eqref{good g}.  Consequently, 
\begin{eqnarray*} 
G_{\tilde{f}}(t, \tilde{a}(t)) &=& \frac12 \, G_{\tilde{f}}(t, \tilde{f}_t(\tilde{a}(t))) \\
	&=& \frac12 \left( G_{\tilde{f}}(t, \tilde{v}(t)) + \log|a(t)-1+t^2| \right) \\
	&=&	\frac12 \, g_{f,v}(t) + \frac12 \, h(t) + \frac12\,  \log|t|
\end{eqnarray*}
so that, by the definition of $g_{f,a}$ in \eqref{G near 0}, we have $\eta=1/2$ and
	$$g_{f,a}(t) = \frac12 \, g_{f,v}(t) + \frac12 \, h(t)$$
on $\D^*$.
We conclude that the function $g_{f,a}$ also fails to be bounded near $t=0$.

\bigskip
\section{Lyapunov exponents and the bifurcation current}
\label{Lyapunov}

The Lyapunov exponent of an individual rational map $f$
on $\P^1$ of degree $>1$, with respect to its unique measure $\mu_f$ 
of maximal entropy 
on $\P^1$, is 
the positive and finite quantity
	$$L(f) = \int_{\P^1} \log |f'| \, d\mu_f,$$
where $|\cdot|$ is any choice of metric on the tangent bundle of $\P^1$. 

Let $f_t$ be a holomorphic family of rational maps on $\P^1$ of degree $d>1$
parameterized by $\D^*$ whose coefficients extend to meromorphic functions 
on $\D$. If all the critical points of $f_t$ 
are parameterized by holomorphic maps $c_1,\ldots,c_{2d-2}:\D\to\P^1$,
then 
\begin{equation} \label{L formula}
	L(f_t) = h(t) + \sum_{j=1}^{2d-2}g_{f,c_j}(t)
\end{equation}
on $\D^*$, for a harmonic function $h$ on $\D^*$ satisfying $h(t) = O(\log|t|)$ as $t\to 0$ \cite[Theorem 1.4]{D:lyap}, \cite[Theorem C]{Favre:degenerations}.  By the symmetry in the critical points in \eqref{L formula}, this formula holds even if the critical points cannot be holomorphically parameterized on $\D^*$.  The {\em bifurcation current} associated to the family $f_t$ can be given by 
	$$T_{\mathrm{bif}} := \frac{1}{2\pi}\, \Delta L(f_t)$$
on $\D^*$, in the sense of distributions; the original definition of $T_{\mathrm{bif}}$ in \cite{D:current} was based on the right hand side of \eqref{L formula}.  From \cite[Theorem 1.1]{D:lyap}, the support of $T_{\mathrm{bif}}$ is equal to the bifurcation locus of the family $f_t$ in the sense of 
\cite{Mane:Sad:Sullivan,Lyubich:stability}.  

In particular, because the sum $\sum_j g_{f,c_j}$ in \eqref{L formula} is $o(\log |t|)$ near $t=0$, we see that the bifurcation current $T_{\operatorname{bif}}$
has a bounded potential if and only if the sum $\sum_j g_{f,c_j}$ is bounded near $t=0$.  

\begin{proof}[Proof of Theorem \ref{discontinuity theorem}.]
We give examples in an arbitrary degree $>1$.  First, let $f_t$ be the holomorphic family of quadratic rational maps on $\P^1$ parameterized by $\D^*$ described in \S\ref{quadratic}. As we have seen, neither of the functions $g_{f,c_{\pm}}$ extend continuously to $\D$; indeed, both tend to $-\infty$ as $t\to 0$.  Hence by \eqref{L formula}, the bifurcation current for the family $f_t$ fails to have a potential bounded around $t=0$.   

Next, let $f_t$ be the holomorphic family of rational maps on $\P^1$ of degree $d>2$ parameterized by $\D^*$, described in \S\ref{cubic}.  As we have seen, the constant map $a(t)\equiv a_0$ on $\D$ satisfies $f_t'(a(t))=0$ for every $t\in\D^*$, and the function $g_{f,a}$ tends to $-\infty$ as $t\to 0$.  Taking an at most finitely ramified holomorphic covering $\pi:\D^*\to\D^*$ if necessary, all the critical points of $f_{\pi(s)}$
are parameterized by holomorphic maps $c_1, \ldots, c_{2d-2}:\D\to\P^1$.  We may assume the points are labeled so that $c_1=a\circ\pi$ on $\D$.   By the formula \eqref{good g} for $g_{f,c_1}$, we have $g_{f_{\pi(\cdot)},c_1}(s)=g_{f,a}(\pi(s))$ on $\D^*$.  On the other hand, for every $j\in\{2,\ldots,2d-2\}$, 
the function $g_{f_{\pi(\cdot)}, c_j}$ is bounded from above on $\{0<|s|\le r\}$
for every $r\in(0,1)$, by Lemma \ref{subharmonic}.  Hence the sum $\sum_j g_{f_{\pi(\cdot)},c_j}(s)$ tends to $-\infty$ as $s\to 0$.  Therefore, the bifurcation current associated to the family $f_t$ fails to have a potential bounded around $t=0$.
\end{proof}

\bigskip
\section{Limitations of the construction}
\label{limitations}

To find the examples of Section \ref{recipe}, 
we used a rational map $\phi\in \C(z)$ of degree $\ge 1$
and points $a_0, h\in \P^1(\C)$ such that
	$$0 < [\phi^{n_j}(a_0), h] < r_{n_j}$$
in the chordal metric $[\cdot,\cdot]$, along a sequence $n_j \to \infty$, 
with $(r_n)$ chosen so that 
	$$\lim_{n\to\infty}\frac{\log r_n}{d^n}=-\infty.$$
Combining \eqref{before choosing r} with \eqref{estimate near 0} guaranteed that 
$\lim_{t\to 0}g_{f,a}(t)=-\infty$.  Looking carefully at the estimates, we see that the orbit $\{\phi^n(a_0)\}$ needs only to satisfy a weaker divergence condition 
\begin{equation} \label{fast}
	\sum_{n=0}^\infty \frac{\log \, [\phi^n(a_0), h]}{d^n}  = -\infty,
\end{equation}
with $\phi^n(a_0) \not= h$ for all $n \in\bN\cup\{0\}$, to achieve our conclusion with this method.

As observed in the Introduction, the function $g_{f,a}$ will always extend continuously to $\D$ when $f_t$ is a family of polynomials, by \cite[Main Theorem]{Favre:Gauthier:continuity}.  Here we explain explicitly why our construction breaks down for polynomials.

\begin{prop} \label{polynomial}
The construction of Section \ref{recipe} cannot produce any pair $(f,a)$ such that for every $t\in\D^*$, $f_t$ is M\"obius conjugate to a polynomial.
\end{prop}

\begin{proof}
Suppose that $f_t$ is a holomorphic family of rational maps of degree $d>1$ 
parameterized by $\D^*$, that for every $t\in\D^*$, 
there exists $A_t\in\mathrm{PSL}(2,\bC)$ such that 
$A_t\circ f_t\circ A_t^{-1}$ is a polynomial, and that
$\lim_{t\to 0}f_t=\phi$ locally uniformly on $\P^1\setminus \{h\}$ 
for some $h\in \P^1$ and some $\phi\in\bC(z)$ of degree $d-1$ $(>0)$.
For every $t\in\D^*$, the point $p_t:=A_t^{-1}(\infty)$ is
a superattracting fixed point of $f_t$ for which $\deg_{p_t}f_t= d$.
We first claim that $\lim_{t\to 0} p_t=h$;
otherwise, there is a sequence $(t_j)$ in $\D^*$ tending to $0$ as $j\to\infty$
such that there is the limit $p:=\lim_{j\to\infty}p_{t_j}\in\P^1\setminus\{h\}$. 
By the locally uniform convergence $\lim_{t\to 0}f_t=\phi$ 
on $\P^1\setminus\{h\}$, $\deg\phi>0$, and the Argument Principle, this $p$ must be a superattracting fixed point of $\phi$ 
for which $\deg_p\phi=d$, contradicting $\deg\phi=d-1$.  We next claim that $\phi^{-1}(h)=\{h\}$; for,
if there is a point $q \in \P^1\setminus\{h\}$ for which $\phi(q) = h$, then by 
the first claim, the locally uniform convergence $\lim_{t\to 0}f_t=\phi$ 
on $\P^1\setminus\{h\}$, $\deg\phi>0$, and the Argument Principle, 
for any $t\in\D^*$ close enough to $0$, there must exist a point 
$q_t\in\P^1\setminus\{p_t\}$ (near $q$) for which 
$f_t(q_t)=p_t$, 
contradicting $\deg f_t=d$.

Suppose $a: \D\to \P^1$ is any holomorphic map with $a(0)=:a_0\neq h$.
If $d>2$ so that $\deg \phi = d-1>1$, then by the second claim, we have
a constant $C < 0$ so that
$$\log \, [\phi^n(a_0), h] \geq C\cdot (d-1)^n$$
for all $n\in\bN$. 
If $d=2$ so that $\deg\phi= d-1=1$, then by the second claim above, 
the orbit $\{\phi^n(a_0)\}$ can accumulate to $h$ and 
satisfy $\phi^n(a_0)\neq h$ for any $n\ge 0$
only if $h$ is an attracting or parabolic fixed point of $\phi$.
Hence we still have a constant $C < 0$ so that
	$$\log \, [\phi^n(a_0), h] \geq C\cdot n$$
for all $n\in\bN$. Therefore in both cases,
the points $a_0, h\in\P^1$ cannot satisfy \eqref{fast}.
\end{proof}

Working over the field $\C$ of complex numbers 
allowed us to exploit the Baire Category Theorem in our construction.  In fact, the construction is impossible over a field such as $\Qbar$.

\begin{prop} \label{Qbar}
The construction of Section \ref{recipe} cannot produce any pair $(f,a)$ 
such that
the map $\phi$ and points $a_0$ and $h$ are simultaneously defined over $\Qbar$.
\end{prop}

\begin{proof}
Suppose $f_t$ is any holomorphic family of rational maps of degree $d>1$ parameterized by $\D^*$ such that $\lim_{t\to 0}f_t=\phi$ locally uniformly on $\P^1\setminus\{h\}$, for some $\phi \in \Qbar(z)$ of degree $d-1$ and some $h\in\P^1(\Qbar)$.  Fix any point $a_0 \in \P^1(\Qbar)$ such that $\phi^n(a_0) \not= h$ for all $n\geq 0$.  

Suppose first that $d>2$, so that $\deg \phi=d-1>1$. 
If 
there is $A\in\mathrm{PSL}(2,\overline{\Q})$ such that
either $A\circ\phi\circ A^{-1}$ or $A\circ\phi^2\circ A^{-1}$ 
is a polynomial and that $A(h)=\infty$, then
we have a constant $C< 0$ so that 
	$$\log \, [\phi^n(a_0), h] \geq C (d-1)^n$$
for all $n\in\bN$.
Otherwise, by \cite[Theorem E]{Silverman:integer}, 
which uses the Roth theorem, we have the stronger result that 
	$$\log \; [\phi^n(a_0), h] = o((d-1)^n)$$
as $n\to \infty$. 
Therefore, in both cases, $\phi$, $a_0$, and $h$ cannot satisfy \eqref{fast}.  

Now suppose that $d = 2$ so that $\deg \phi=d-1 = 1$. 
Note that the orbit $\{\phi^n(a_0)\}$ can accumulate to $h$ and 
satisfy $\phi^n(a_0)\neq h$ for any $n\ge 0$
only if either $h$ is an attracting or parabolic fixed point of $\phi$ (in $\P^1(\overline{\Q})$) or there exists $A\in\mathrm{PSL}(2,\overline{\Q})$
such that $A\circ\phi\circ A^{-1}$ is an irrational rotation $z\mapsto\lambda z$, 
where $\lambda$ is not a root of unity, $\lambda\in\overline{\Q}$, 
and $|\lambda|=1$, with $|A(h)|=|A(a_0)|=1$.  
In the former case, we have a constant $C < 0$ so that
$\log \, [\phi^n(a_0), h] \geq C\cdot n$
for all $n\in\N$. So $\phi$, $a_0$, and $h$ cannot satisfy \eqref{fast}. 

In the latter case, we claim that we still have a constant $C<0$ such that
$$
\log[\phi^n(a_0),h]\ge C\cdot n
$$
for all $n\in\N$; since $A\in\mathrm{PSL}(2,\overline{\Q})$ is biLipschitz
with respect to $[\cdot,\cdot]$, we can assume that
$\phi$ is an irrational rotation 
$z\mapsto\lambda z$, where $\lambda$ is not a root of unity, 
$\lambda\in\overline{\Q}$, 
and $|\lambda|=1$, with $h,a_0\in\overline{\Q}$ and $|h|=|a_0|=1$. 
Fix a number field $K$ so that $\lambda,a_0,h\in K$, 
and denote by $M_K$ the set of all places (i.e., equivalence classes of
non-trivial either archimedean or non-archimedean absolute values) of $K$. 
Recall that there are a family
$(N_v)_{v\in M_K}$ in $\bN$ and a family $(|\cdot|_v)_{v\in M_K}$ of representatives
$|\cdot|_v$ of places $v$ such that
for every $x\in K^*$, $|x|_v=1$ for all but finitely many $v\in M_K$
and $\prod_{v\in M_K}|x|_v^{N_v}=1$.
Then by the (strong) triangle inequality, we can choose
a family of real numbers $C_v\ge 1$, $v\in M_K$, such that 
$|\lambda^n - h/a_0|_v\leq C_v^n$ for any $v\in M_K$ and any $n\in\N$
and that $C_v = 1$ for all but finitely many $v\in M_K$.
We also note that $\lambda^n a_0=\phi^n(a_0)\neq h$
for all $n\in\N$. Hence for every $v_0\in M_K$ and every $n\in\N$,
we have $|\lambda^n - h/a_0|_{v_0}\ge(\prod_{v\in M_K} C_v^{-N_v})^n$.
In particular, recalling that $[z,w]=|z-w|[z,\infty][w,\infty]$
on $\bC\times\bC$, there is a constant $C<0$ such that
$$
\log[\phi^n(a_0),h]=\log|\lambda^n - h/a_0|-\log 2
\ge C\cdot n
$$
for all $n\in\N$. So the claim holds, and $\phi$, $a_0$, and $h$ cannot satisfy \eqref{fast}. 
\end{proof}

\bigskip \bigskip

\def\cprime{$'$}

\end{document}